\documentclass[10pt]{amsart}
\usepackage{amsthm, amssymb, amsmath, graphicx, subfig}
\usepackage{textcomp, extraipa, enumerate}
\usepackage{url}
\graphicspath{{./Graphics/}}

\newtheorem{thm}{Theorem}
\newtheorem{cor}{Corollary}
\newtheorem{prop}{Proposition}
\newtheorem{lemma}{Lemma}
\newtheorem{define}{Definition}
\newtheorem*{note}{Note}
\newtheorem{conj}{Conjecture}

\def\gst{\sliding{$g_4$}}
\def\g{\sliding{$g_c$}}
\def\THg{\underrightarrow{g_c}}
\def\THgst{\underrightarrow{g_4}}
\def\threefive{xT_{2,3} + yT_{2,5}}
\def\Cq{\mathcal{C}_{\mathbb{Q}}}

\def\jump{\mathrm{j}}
\def\mult{\mathrm{mult}}

\title{The Stable Concordance Genus}
\author{M. Kate Kearney}

\begin{document}

\begin{abstract} The concordance genus of a knot is the least genus of any knot in its concordance class.  Although difficult to compute, it is a useful invariant that highlights the distinction between the three--genus and four--genus.  In this paper we define and discuss the stable concordance genus of a knot, which describes the behavior of the concordance genus under connected sum.   \end{abstract}

\maketitle

\section{Introduction}

	The concordance genus of a knot, $g_c$, is the least three-genus of a knot concordant to the knot.  That is, $g_c(K) = min\{g_3(J) | J \sim K\}$.  The concordance genus is bounded below by the four-genus and bounded above by the three-genus.  This makes the concordance genus a valuable tool to describe the difference between these invariants.  In simple cases the concordance genus is not difficult to calculate, since there are a variety of algebraic tools that give bounds for the concordance genus.  In \cite{KK:11Crossing}, the author has given the calculation of the concordance genus for all but 19 of the 552 prime 11--crossing knots.  Unfortunately, as the crossing number increases, it becomes increasingly difficult to find concordances.  In order to study a broader picture of the behavior of the concordance genus, one then considers knots which are not prime.  The stable concordance genus, defined in this paper, describes the behavior of the concordance genus of a given knot under connect sum.  Section one of this paper will give the definition and basic properties of the stable concordance genus.  The remainder of the paper will discuss examples of calculations.  In particular, a formula is given for sums of certain $(2,n)$ torus knots.  Finally, we discuss applications to the study of concordance, as well as open questions related to the stable concordance genus.  

{\sl Acknowledgements}  Gratitude is owed to Chuck Livingston and Pat Gilmer for many helpful conversations on this topic.	
	
\section{Definition and Properties}

	The stable four-genus, $\gst$, is discussed by Livingston in \cite{L:Stable4Genus} (notated as $g_{st}$ in Livingston's work).  It is defined as $\gst(K) := \displaystyle\lim_{n\rightarrow\infty} \frac{g_4(nK)}{n} $, where $g_4$ is the four-genus of the knot.  We define the stable concordance genus, $\g$, similarly in terms of $g_c$, the concordance genus.

\begin{define}
	 $\THg(K) := \displaystyle\lim_{n\rightarrow\infty}\frac{g_c(nK)}{n} .$
\end{define}

That this is well-defined is an immediate consequence of Theorem 1 of Livingston's paper \cite{L:Stable4Genus}, but we also include a direct proof here.

\begin{prop}
The stable concordance genus is well-defined.
\end{prop}

\begin{proof}
We first observe that $g_c(K)$ is subadditive (that is, for all $K$ and $J$, $g_c(K \# J) \leq g_c(K) + g_c(J)$).  In particular, $g_c(nK) \leq ng_c(K)$ for all $n$, and for all $n$ and $m$, $$\frac{g_c(nmK)}{nm} \leq \frac{ng_c(mK)}{nm} = \frac{g_c(mK)}{m}.$$ 

Furthermore, $g_c(K)$ is non-negative, and hence bounded below.  Let $L$ be the greatest lower bound of $\{ \frac{g_c(nK)}{n}\}_{n \in \mathbb{Z}_+}$.  Then for all $n$, $\frac{g_c(nK)}{n} > L$, and for any $\epsilon >0$ there is an $N$ such that $\frac{g_c(NK)}{N} \leq L + \frac{\epsilon}{2}$.  For any $n$, we may write $n=aN + b$, where $0 \leq b < N$.  Let $B = \max\{g_c(bK)\}_{0 \leq b < N}$.  Then for each $0 \leq b < N$, $g_c(bK) \leq B$.  By subadditivity, $g_c(nK) \leq ag_c(NK) + g_c(bK)$, so 
$$\frac{g_c(nK)}{n} \leq \frac{ag_c(NK)}{aN+b} + \frac{g_c(bK)}{aN+b} \leq \frac{g_c(NK)}{N} + \frac{B}{aN}.$$  

Then if $n$ is large enough such that $\frac{B}{aN} \leq \frac{\epsilon}{2}$ (and as mentioned above, we have $\frac{g_c(NK)}{N} \leq L + \frac{\epsilon}{2}$), then 
$$L \leq \frac{g_c(nK)}{n} \leq \frac{g_c(NK)}{N} + \frac{B}{aN} \leq L + \frac{\epsilon}{2} + \frac{\epsilon}{2} = L + \epsilon.$$ 
So $\displaystyle\lim_{n \rightarrow \infty} \frac{g_c(nK)}{n}$ exists and is equal to $L$.
\end{proof}

By a similar argument, we see that the stable concordance genus is multiplicative: 
 $$\g(mK) = \lim_{n \rightarrow \infty} \frac{g_c(nmK)}{n} = \lim_{nm \rightarrow \infty} \frac{g_c(nmK)}{\frac{nm}{m}} = m \lim_{nm \rightarrow \infty} \frac{g_c(nmK)}{nm} = m \g(K).$$
The stable concordance genus is first defined for knots, but any two concordant knots have the same concordance genus, and consequently the same stable concordance genus, so we can consider $\g$ to be a function on the concordance group, $\mathcal{C}$.  We can extend $\g$, by multiplicativity, to be defined on $\Cq = \mathcal{C} \otimes \mathbb{Q}$.  

Since the concordance genus is subadditive, the stable concordance genus is also subadditive.  Although it is not strictly positive, $\g$ is at least non-negative.  Hence the stable concordance genus (like the stable four-genus \cite{L:Stable4Genus}) is a seminorm.  That is, it is a non-negative function which is multiplicative and subadditive.  Consequently, $\g$ satisfies a triangle inequality.  

We aim to understand $\g$ by looking at its unit ball, $B_{stc} = \{K \in \mathcal{C} | \g(K) = 1\}$ (similarly $B_{st4}$, the unit ball for the stable four-genus) and particularly the restriction to two-dimensional subspaces.  We will explore some basic examples of computation, with the goal of finding the unit ball of collections of knots of the form $xT_{2,n} + yT_{2,m}$.

To prepare to calculate $\g$ for basic examples, we first observe several properties of the invariant.  Detailed definitions of the Alexander polynomial and classical knot signature can be found in many sources, such as \cite{Rolfsen}, and as such are omitted here.  Instead we simply observe several useful properties.
\begin{itemize}
\item $g_3 \geq g_c \geq g_4 \geq \frac{1}{2}|\sigma|$. 
\item Consequently, $\g \geq \gst \geq \frac{1}{2}|\sigma|$ (recall that $\sigma(K \# J) = \sigma(K) + \sigma(J)$).
\item  The same inequalities hold for the Tristram-Levine signatures, so in fact $\g \geq \frac{1}{2}|\sigma_t(K)|$ for $t \in [0,1]$.
\item Further, $B_{stc} \subset B_{st4} \subset B_{\sigma}$ (where $B_{\sigma}$ is the region in which the Tristram-Levine signatures all have values of two or less).
\item $g_3(K) \geq \frac{1}{2}\deg(\Delta_K(t))$, the degree of the Alexander polynomial of $K$.  
\end{itemize}
Note that the Tristram-Levine signature is defined to be $\sigma_t(K) = \lim_{\epsilon \rightarrow 0} \frac{1}{2}(\sigma'_{t-\epsilon}(K) + \sigma'_{t+\epsilon}(K))$ where $\sigma'_t(K)=\text{signature}((1-e^{2\pi i t})V+(1-e^{-2\pi i t})V^T)$ and $V$ is a Seifert matrix for $K$.

\begin{thm}[Fox-Milnor]
If K is slice, then $\Delta_K(t) = f(t)f(t^{-1})$ for some polynomial $f(t)$. \cite{FM:CobordismOfKnots}
\end{thm}
As a consequence, if K is concordant to J, then $\Delta_K(t)\Delta_J(t) = f(t)f(t^{-1})$ for some polynomial $f(t)$.  
So if we can write $\Delta_K(t)=f(t)f(t^{-1})g(t)$, for some polynomials $f(t)$ and $g(t)$, where $g(t)$ has no factors of the form $h(t)h(t^{-1})$, then we may conclude $g_c(K) \geq \frac{1}{2}\deg(g(t))$.  In Section 3 we observe that in conjunction with jumps in the signature function, we can use this lower bound for the concordance genus to also bound $\g$ from below.

\section{Preliminary Examples}

To begin exploration of $\g$ we calculate values for prime knots with eight crossings or fewer.  In the following discussion, values of classical invariants including signature and Alexander polynomial are as given on KnotInfo \cite{knotinfo}.

\bigskip
$\mathbf{3_1}$:  As observed in the previous section, since $g_c(nK) \geq \frac{1}{2}|\sigma(nK)|$ for all $n$, we know $\g(K) \geq \frac{1}{2}|\sigma(K)|$.  We check, $|\sigma(3_1)| = 2$, so this gives us $\g(3_1) \geq 1$.  On the other hand, $g_c(n3_1) \leq ng_c(3_1)$ for all $n$, so $g_c(n3_1) \leq n$ for all $n$.
$$1 \leq \g(3_1) = \lim_{n\rightarrow\infty} \frac{g_c(n3_1)}{n} \leq \lim_{n\rightarrow\infty} \frac{n}{n} = 1$$
So $\g(3_1) = 1$.
\bigskip

Generalizing the basic algebra used in the previous calculation, we observe the following:

\begin{prop}
The stable concordance genus is bounded above by the concordance genus.
\end{prop}

\begin{prop}\label{signprop}
If $\frac{1}{2}|\sigma(K)| = g_c(K)$, then $\THg(K) = g_c(K) = \frac{1}{2}|\sigma(K)|$.
\end{prop}

Proposition \ref{signprop} applies to 15 other prime knots of eight or fewer crossings: $$5_1, 5_2, 6_1, 7_2, 7_3, 7_4, 7_5, 8_8, 8_9, 8_{10}, 8_{11}, 8_{15}, 8_{19}, 8_{20}.$$  This includes the slice knots.  In fact, as a special case of Proposition \ref{signprop}, all slice knots have stable genus zero.

\begin{cor}
If $K$ is slice, $K$ is stably slice (that is, $\THg = \THgst = 0$).
\end{cor}

\bigskip
$\mathbf{4_1}$:  The figure eight knot is amphichiral, so $g_4(2*4_1) = g_4(4_1 \# -4_1) = 0$.  Therefore $g_4(2n4_1) = g_c(2n4_1)= 0$.  Using the fact that the stable concordance genus is well-defined, 
$$\g(4_1) = \lim_{n\rightarrow\infty}\frac{g_c(n4_1)}{n} = \lim_{k\rightarrow\infty}\frac{g_c(2k4_1)}{2k} = \lim_{k\rightarrow\infty}0 = 0$$
So $\g(4_1) = 0$.  Since $g_c(4_1) = 1$, this is an example for which $\g(K) \neq g_c(K)$.
\bigskip

In fact, we can use the same technique to see that the stable concordance genus vanishes for all knots for which $g_c(nK) = 0$ for some $n$.  This is exactly the knots of finite order in $\mathcal{C}$.

\begin{prop}
Any knot which has finite order in $\mathcal{C}$ is stably slice.  In particular, amphichiral knots are stably slice.
\end{prop}

For prime knots of eight or fewer crossings, this applies to $$6_3, 8_3, 8_{12}, 8_{17}, 8_{18}$$
as well as several of the previously mentioned knots including the slice knots.

\bigskip
$\mathbf{6_2}$:  We begin by checking the signature and Alexander polynomial.  The signature is $\sigma(6_2) = -2$, and $\Delta_{6_2}(t) = 1-3t+3t^2-3t^3+t^4$, which is irreducible in $\mathbb{Z}[t,t^{-1}]$.  The concordance genus is $g_c(6_2) = 2$.  So we have $1 = \frac{1}{2}|\sigma(K)| \leq \g(6_2) \leq g_c(6_2) = 2$.  The Tristram-Levine signature jumps at the two complex roots of $\Delta_{6_2}(t)$, $\alpha$ and $\overline{\alpha}$ by two.  While we cannot get a stronger bound directly from the Tristram-Levine signatures, we can use the jump function to show that half the degree of the Alexander polynomial of $6_2$ (or in similar cases, a factor of the Alexander polynomial) does bound $\g$.  We'll pause our calculation to discuss Proposition \ref{jumpprop} and it's consequences.

Let $\jump_{\rho}(K)$ denote the jump in the signature function of $K$ at $\rho$.  The following lemma is the key ingredient to proving Proposition \ref{jumpprop}, as we will see below.  This lemma has been stated by Garoufalidis \cite{G:JonesSignature}, although a complete proof is not given in the literature.

\begin{lemma}\label{jumplemma}
If $\rho$ is a root of the Alexander polynomial on $S^1$, then
$|\jump_{\rho}(K)|=2 \, a_{\rho}$, where 
\begin{enumerate}[(a)]
\item  $a_{\rho}$ is an integer 
\item  $a_{\rho} \leq \mult(\rho,\Delta_K(t))$, where $\mult(\rho,\Delta_K(t))$ is the
multiplicity of $\rho$ in $\Delta_K(t)$, and
\item   $a_{\rho} \equiv \mult(\rho,\Delta_K(t)) \bmod 2$.
\end{enumerate}
Moreover, $\jump_{\rho}(K)=-\jump_{\bar{\rho}}(K)$.
\end{lemma}

The most direct proof of part (b) that the author is aware of uses Milnor's definition of $\sigma_{\theta}$ signatures \cite{M:InfiniteCyclicCoverings}.  These are equivalent to the jump function defined above, as shown by Matumoto \cite{M:SignatureInvariants}.  Milnor describes how to split $H^1(\tilde{X}, \partial \tilde{X})$ into $[p(t)]$-primary summands in $\mathbb{R}[t,t^{-1}]$, and accordingly split the quadratic form, so as to separately analyze the contributions of each factor, $p(t)$, of the Alexander polynomial to the signature.  A careful analysis of the dimensions of these summands produces the desired inequality $a_{\rho} \leq \mult(\rho,\Delta_K(t))$, where $\mult(\rho,\Delta_K(t))$ is the multiplicity of $\rho$ in $\Delta_K(t)$.

An alternative proof is given by considering the $\mathbb{Q}(t)$ Hermitian form given by $B_t =  (1-t^{-1})V+(1-t)V^t$.  Notice that that $(1-t^{-1})^n\Delta_K(t) = \det(B_t)$.  The matrix $B_t$ can be diagonalized.  In particular, there is a matrix $A$ with $\det(A) = 1$ and $ABA^*$ is diagonal, where $A^*$ is the conjugate transpose.  Chosing $A$ carefully, one can insist that the diagonal matrix $ABA^*$ has rational functions on the diagonal, and avoid having factors $p(t)$ of the Alexander polynomial as denominators.  Then we can see a direct relationship between jumps in the signature function and factors of the Alexander polynomial.  An inductive proof on dimension shows that the matrix can in fact be diagonalized in such a way.  

Part (a) is an immediate consequence of the fact that knot signatures are always even.  Part (c) follows from considering the diagonalized matrix from the second proof of part (b).  The details of the proof are left to the reader.

\begin{prop}\label{jumpprop}
If a knot, $K$, has Alexander polynomial $\Delta_K(t)=f(t)^xg(t)$ and $\jump_\rho(K) = \pm 2x$ for where $f(t)$ is the minimal polynomial for $\rho$ in $\mathbb{Z}[t, t^{-1}]$, then for any $J$ concordant to $K$, $f(t)^x$ is a factor of $\Delta_J(t)$.
\end{prop}

\begin{proof}  This is an immediate consequence of Lemma \ref{jumplemma}.

In particular, for $f(t)$ the minimal polynomial of $\rho$ with $\jump_\rho(K) = \pm 2x = \jump_\rho(J)$ (since the signature function is a concordance invariant), then $x \leq \mult(\rho, \Delta_J(t))$ and hence $f(t)^x$ is a factor of $\Delta_J(t)$.
\end{proof}

To clarify the application of this proposition, we define two new polynomials.

\begin{define}
The concordance polynomial of a knot, $K$, is the maximal degree polynomial which divides the Alexander polynomial of all knots concordant to $K$.  We will denote it $\Delta_K^c(t)$.
\end{define}

This is well-defined up to multiplication by $\pm t^k$.  Notice that $\Delta_K^c(t)$ divides $\Delta_J(t)$ for all $J \sim K$.  Since $\Delta_K^c(t)$ divides $\Delta_K(t)$ in particular, we see that $\Delta_K^c(t)$ is simply a product of the factors of $\Delta_K(t)$ which also divide each $\Delta_J(t)$ for $J \sim K$.

\begin{define}
The jump polynomial of a knot, $K$, is given by $$\Delta_K^j(t) := \prod_{f_i(t)} f_i(t)^{\jump_i(K)}$$  where $f_i$ are the irreducible factors of $\Delta_K(t) = \prod f_i(t)^{x_i(t)}$, and $$\jump_i(K) := \max \{ |\frac{1}{2} \jump_{\alpha}(K)| : \alpha \text{ is a root of }f_i(t) \}.$$
\end{define}

The following are immediate consequences of these definitions and the previous results.

\begin{prop}
The jump polynomial of $K$ divides the concordance polynomial of $K$, and both divide the Alexander polynomial of $K$.  In particular, $$\deg (\Delta_K^j(t)) \leq \deg (\Delta_K^c(t)) \leq \deg (\Delta_K(t))$$
\end{prop}

\begin{prop}
The concordance polynomial is a concordance invariant.  Furthermore, $\frac{1}{2} \deg(\Delta_K^c(t)) \leq g_c(K)$.
\end{prop}

\begin{proof}
We observed above that $\Delta_K^c(t)$ divides $\Delta_J(t)$ for all $J \sim K$.  Consequently $\deg(\Delta_K^c(T)) \leq  \deg(\Delta_J(t))$ for all $J \sim K$, so in particular, $\deg(\Delta_K^c(T)) \leq  \min\{\deg(\Delta_J(t)): J \sim K\}$.  Then since $\frac{1}{2} \deg(\Delta_J(t)) \leq g_3(J)$ for each $J \sim K$, we have $$\frac{1}{2} \deg(\Delta_K^c(t)) \leq \min \{ \frac{1}{2} \deg(\Delta_J(t)) : J \sim K \} \leq \min \{g_3(J) : J \sim K \} = g_c(K).$$
\end{proof}

\begin{prop}\label{degreebound}
The degree of the jump polynomial is exactly the sum $$\deg \Delta_K^j(t) = \sum_i (\deg f_i(t) ) * j_i(K)$$ where $f_i(t)$ and $j_i(K)$ are as given in the definition of the jump polynomial.  Moreover, one half of this value is a lower bound for the concordance genus of $K$.
\end{prop}

$\mathbf{6_2, continued}$:  Now, let's return to the example at hand.  In this case, $\sigma_{\omega}(n6_2)$ jumps by $-2n$ at $\alpha$ (and $\overline{\alpha}$), i.e. $\jump_{\alpha}(n6_2) = -2n$.  The degree of the corresponding irreducible factor (which is in this case $\Delta_{6_2}(t)$) is 4.  Hence by Proposition \ref{degreebound}, for all $n \geq 1$, $$2n =\frac{1}{2}\deg(1-3t+3t^2 -3t^3 +t^4)^{n} \leq g_c(n6_2).$$

We will reiterate to clarify the details.  Since $\Delta_{6_2}(t)$ is the minimal polynomial for $\alpha$, we may conclude from Proposition \ref{jumpprop} that $\Delta_{n6_2}(t) = (\Delta_{6_2})^n = (\min(\alpha))^n$ is a factor of $\Delta_J(t)$ for any $J$ concordant to $n6_2$.

Now, if $\Delta_J(t)$ has $(\Delta_{6_2}(t))^n$ as a factor, we know $\deg (\Delta_J(t)) \geq 4n$ for all $n$, and all $J$ concordant to $n6_2$.  Hence $2n \leq \frac{1}{2} \deg(\Delta_J(t)) \leq g_c(n6_2)$ for all $n \geq 1$, as we saw above.

Finally, we see that
$$ 2 = \frac{\frac{1}{2} \deg (\Delta_{6_2}(t)^n)}{n} \leq \g(6_2) \leq g_c(6_2) = 2.$$
We conclude that $\g(6_2) = 2$.

\bigskip
$\mathbf{8_5}$:
We can also apply Proposition \ref{degreebound} to Alexander polynomials which are products of several irreducible factors.  In this case, to get a sharp bound we require that the signature function jump at roots of each factor of $\Delta_{8_5}(t)$.  The Alexander polynomial of $8_5$ is $\Delta_{8_5}(t) = 1-3t+ 4t^2-5t^3+ 4t^4-3t^5+ t^6 = (1-t+t^2)(1-2t+t^2-2t^3+t^4)$.  The signature functions jumps by $2$ at $\alpha$, the root of $1-t+t^2$, and also by 2 at $\beta$, the root of $1-2t+t^2-2t^3+t^4$.  Hence by a similar argument to above, applied to both factors, and we may conclude that $$3n = \frac{1}{2}[\deg(1-t+t^2)^{n} + \deg(1-3t+3t^2 -3t^3 +t^4)^{n}] \leq g_c(n6_2).$$
So we have $3 = \frac{1}{2}\deg(\Delta_{8_5}(t)) \leq \g(8_5) \leq g_c(8_5) = 3$.

We can similarly calculate the stable concordance genus of $$7_6, 8_2, 8_4, 8_6, 8_7, 8_{14}, 8_{16}.$$
In each of these cases the stable concordance genus is equal to the concordance genus.  There are four prime knots of eight or fewer crossings for which the stable concordance genus is as of yet undetermined: $7_7, 8_1, 8_{13}, 8_{21}$.

\section{Torus Knots}

The stable concordance genus is particularly interesting when we use it to examine larger collections of knots under connect sum.  Here we will look for the stable concordance genus unit ball restricted to sums of the form $xK + yJ$, with $K$, $J$ torus knots.  Having calculated the stable concordance genus of $T_{2,3} = 3_1$ and $T_{2,5} = 5_1$ in the previous section, we begin with sums of these two knots.

\bigskip
$\mathbf{\threefive}$:  The signature function of $\threefive$ jumps at $1/10$, $1/6$, and $3/10$ in $[0,1/2]$, taking on the values: $0 \in [0,1/10)$, $2y \in [1/10, 1/6)$, $2x + 2y \in [1/6, 3/10)$, and $2x + 4y \in [3/10, 1/2]$.  We will first look at the stable four-genus for this family of knots.  The signature function gives us the bounds 

\begin{align*} 	
   & \gst \geq |y| \\
   & \gst \geq |x+y| \\
   & \gst \geq |x + 2y|
\end{align*} 

Considering each of these inequalities for $\gst \leq 1$, we bound a region in the plane (this is the signature ball $B_{\sigma}$ mentioned in section 2).  We then check the corner points of this region, and see that since $\gst(T_{2,3}) = 1$, $\gst(-T_{2,3} + T_{2,5}) = 1$ and $\gst(-2T_{2,3} + T_{2,5}) = 1$.  Since $\gst(xK + yJ) = \gst(-xK -yJ)$ this is enough to determine that this region is in fact the unit ball for the stable 4--genus (Fig \ref{fig:g_stBall}).  Although it does not represent a corner point, $\gst(T_{2,5}) = 2$ as we saw earlier, which is consistent with this calculation.

\begin{figure}[h]
		\centering
		\includegraphics[width=.7\textwidth]{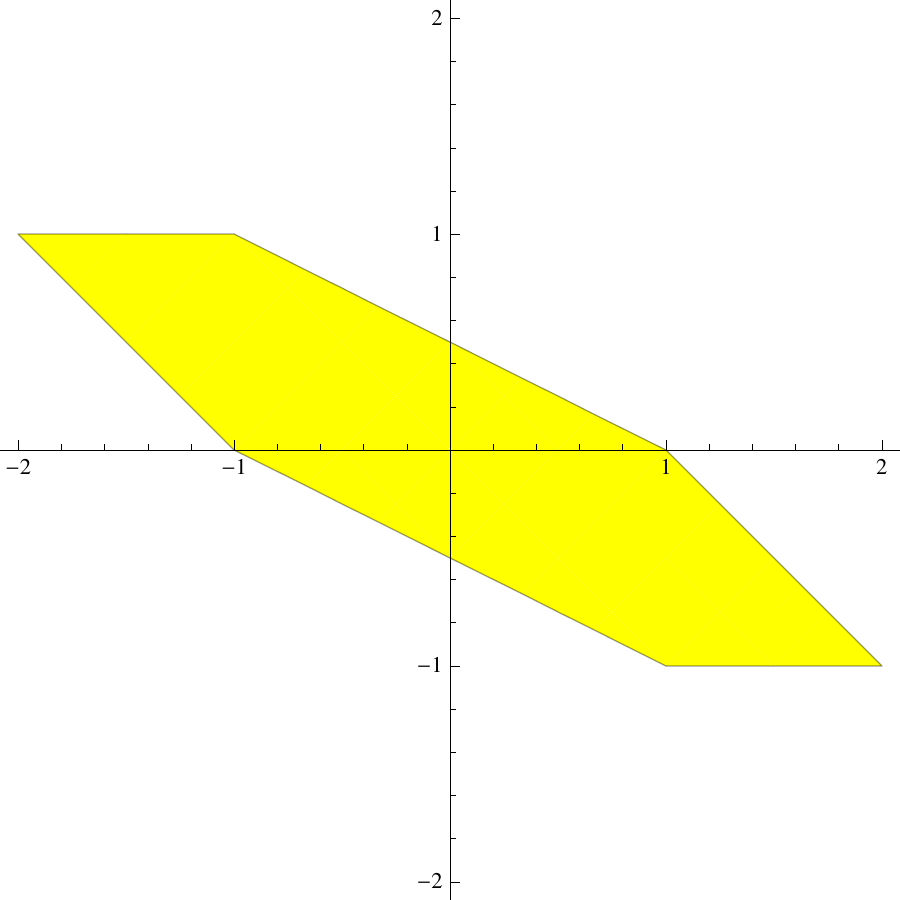}
		\caption{Stable four-genus unit ball for $\threefive$}
	\label{fig:g_stBall}
\end{figure}

To calculate the stable concordance genus unit ball, we generalize the calculation given in the previous section for the knot $6_2$.  Notice that $\Delta_{\threefive}(t) = (1-t+t^2)^{|x|}(1-t+t^2-t^3+t^4)^{|y|}$ and since $\sigma_{\omega}(t)$ jumps at the roots of each factor by $2x$ and $2y$, the jump polynomial is $\Delta_{\threefive}^j = \Delta_{\threefive}$.  
So, by Proposition \ref{degreebound}, we have $\g(\threefive) \geq |x| + 2|y|$.  We know then the unit ball for $\g(\threefive)$ is contained in the ball defined by these equations, but furthermore, $\g(T_{2,3}) = 1$, and $\g(T_{2,5}) = 2$, so by linearity this is the unit ball (Fig \ref{fig:g_stcBall}).

\begin{figure}[h]
		\centering
		\subfloat[Stable concordance genus unit ball]{\includegraphics[width=.5\textwidth]{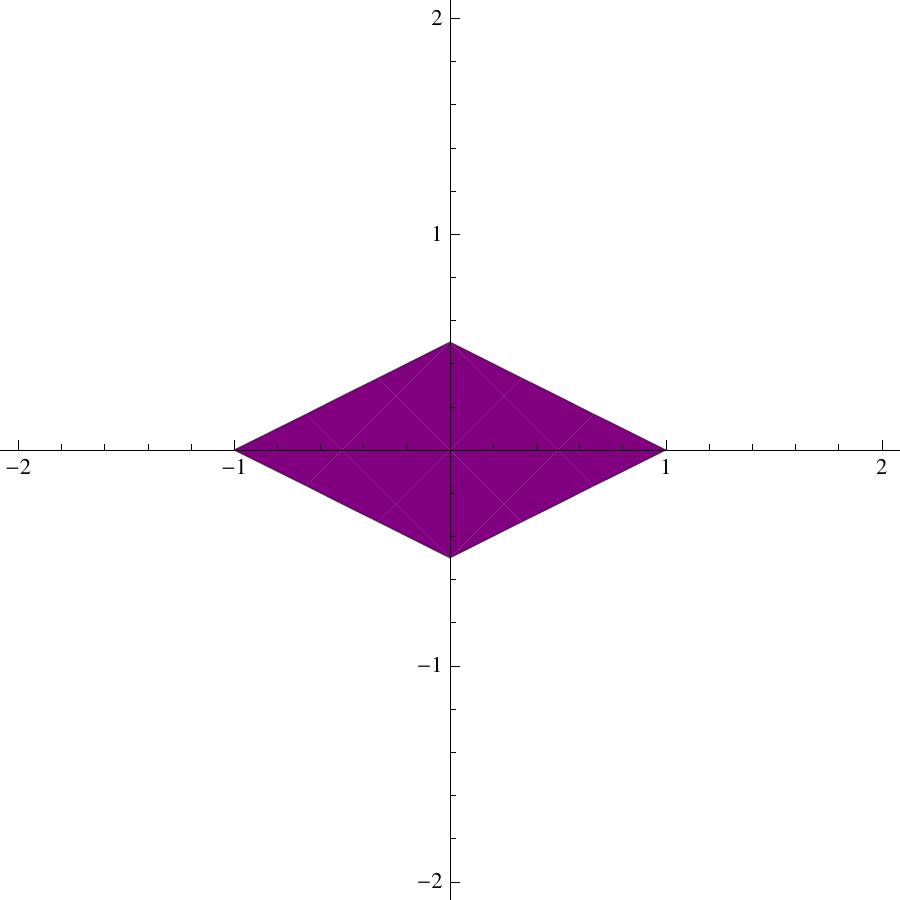}}
		\subfloat[Stable concordance genus unit ball overlaid stable four-genus unit ball]{\includegraphics[width=.5\textwidth]{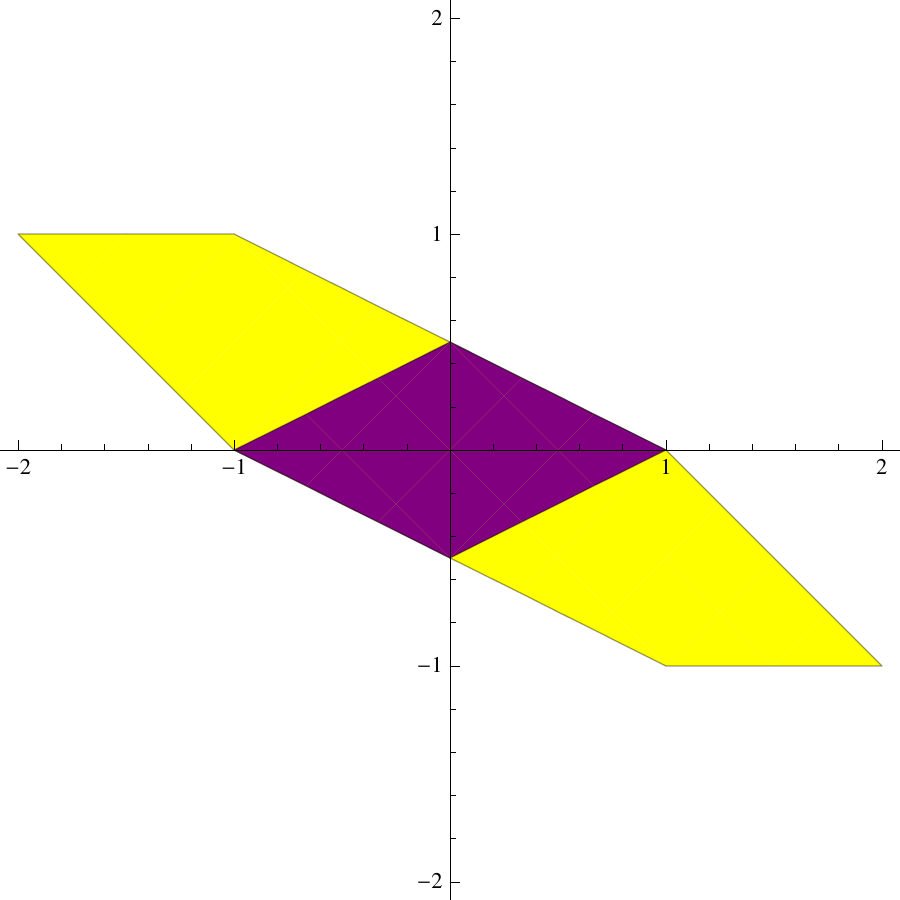}}
		\caption{The stable concordance genus unit ball for $\threefive$ is different than the stable four-genus unit ball.}
	\label{fig:g_stcBall}
\end{figure}

We observe in particular that the unit ball for the stable four-genus is different than the unit ball for the stable concordance genus.  This is the primary observation necessary to prove Theorem \ref{comparisontheorem}. 

\bigskip
$\mathbf{T_{2,n}}$:  For any torus knot of the form $T_{2,n}$, $\g(T_{2,n}) = \frac{n-1}{2}$.  For any such knot, there is a surface of genus $\frac{n-1}{2}$ whose boundary is the knot, so $\g(T_{2,n}) \leq g_3(T_{2,n}) \leq \frac{n-1}{2}$.  On the other hand, $\sigma(T_{2,n}) = \frac{n-1}{2}$, so $\g(T_{2,n}) \geq \gst(T_{2,n}) \geq \frac{n-1}{2}$.  This will assist us in a general calculation of knots of the form $xT_{2,n} + yT_{2,m}$.  We conclude that $\g(T_{2,n}) = \frac{n-1}{2}$.

\bigskip
$\mathbf{xT_{2,n} + yT_{2,m}}$:  As long as the Alexander polynomials of $T_{2,n}$ and $T_{2,m}$ have distinct factors, then the argument for $\threefive$ follows through with any family of knots of the form $xT_{2,n} + yT_{2,m}$.  We claim the following:

\begin{thm}\label{torusknotgenus}
The stable concordance genus of knots of the form $xT_{2,n} + yT_{2,m}$ is $$\frac{n-1}{2}|x| + \frac{m-1}{2}|y|$$ for any $n, m \in \textbf{Z}$ with $n < m$, $kn \neq m$.
\end{thm}

\begin{proof}
The Alexander polynomial for these knots is $\Delta_{T_{2,n}}(t) = \frac{(t^{2n}-1)(t-1)}{(t^2-1)(t^n-1)}$ (recall further that $\Delta_{K \# J}(t) = \Delta_K\Delta_J$).  Notice, $\Delta_{T_{2,n}}$ is a product of the cyclotomic polynomials $\Phi_i(t)$ for $i$ a factor of $2n$ other than $2$ or $n$. Since $n$ is odd, $\Phi_{2n}(t)$ is a factor of $\Delta_{T_{2,n}}(t)$.  And, if $n \neq mk$ then $\Phi_{2n}(t)$ is not a factor of $\Delta_{T_{2,m}}(t)$.  Recall that $\Phi_{2n}(t)$ is the minimal polynomial of the $2n^{th}$ primitive roots of unity, $\xi_{2n}^k$ (for $1 \leq k < n$ and $\gcd(k, n) = 1$).

For each of the primitive $2n^{th}$ roots of unity, the signature function jumps (specifically, $\jump_{\xi_{2n}}(xT_{2,n} + yT_{2,m}) = \pm 2x$ and $\jump_{\xi_{2m}}(xT_{2,n} + yT_{2,m}) = \pm 2y$).  Thus, so as long as $n$ and $m$ are have factors distinct from the other (that is, $n \neq mk$ and $m \neq nk$), then the signature function for $K$ jumps at a root of $\Delta_{T_{2,n}}(t)$ which is not a root of $\Delta_{T_{2,m}}(t)$ and vice versa.  
Applying Proposition \ref{degreebound}, we have $\g(xT_{2,n} + yT_{2,m}) \geq \frac{n-1}{2}|x| + \frac{m-1}{2}|y|$.  As computed above, $\g(T_{2,n}) = \frac{n-1}{2}$.  We conclude then that the unit ball for such knots is defined by the inequality $1 \geq  \g(xT_{2,n} + yT_{2,m}) \geq \frac{n-1}{2}|x| + \frac{m-1}{2}|y|$. Hence, more generally, $\g(xT_{2,n} + yT_{2,m}) = \frac{n-1}{2}|x| + \frac{m-1}{2}|y|$.

\end{proof}

\section{Further discussion}
We observed in the case of $\threefive$, the unit ball for the stable four-genus is not the same as that for the stable concordance genus.  A similar calculation in the more general case of $xT_{2,n}+yT_{2,m}$ allows us to construct examples that show the following theorem.  

\begin{thm} \label{comparisontheorem}
For any $j$, $k \in \mathbb{Q}$, for which $1 \leq j \leq k$, there is some $K \in \Cq$ for which $\THgst(K) = j$, $\THg(K) = k$.  Furthermore,  if $K \in \mathcal{C}$, given any $l \geq k$ then for some knot $K'$ in the concordance class of $K$, $g_3(K') = l$. 
\end{thm}

\begin{proof}
We will use the fact, from Theorem \ref{torusknotgenus}, that the stable concordance genus of the sum of torus knots $xT_{2,2n+1} + yT_{2,2m+1}$ is $n|x| + m|y|$ (for $n, m \in \textbf{Z}$ with $n < m$, $k(2n+1) \neq (2m+1)$).  We also will use that the stable four genus calculation from the previous section generalizes to connect sums of knots $T_{2, n}$.  According to a calculation by Rick Litherland, for these knots the stable four-genus is determined by the signature function \cite{L:Stable4Genus}.  In particular, for $-2 - \frac{c+ 2n}{m-nc} \leq x \leq 0$ (where $c = \lfloor \frac{2m+1}{2n+1} \rfloor$) the upper boundary of the stable four-genus unit ball, restricted to the plane $xT_{2,2n+1} + yT_{2,2m+1}$, is given by the line $nx+my = 1$.  By multiplicativity, if we choose a pair $(x,y)$ on this line, which satisfies $\frac{k}{j} = n(-x) + my = n|x| + m|y|$, then for $K = jxT_{2,2n+1}+jyT_{2,2m+1}$, we have that $\g(K) = k$ and $\gst(K) = j$.  A simple linear algebra computation shows us that $(x,y)$ should be $(\frac{1}{2n}(1-\frac{k}{j}), \frac{1}{2m}(1+\frac{k}{j}))$.  Then by choosing $n$ and $m$ sufficiently large, so that $-2 - \frac{c+ 2n}{m-nc} \leq x \leq 0$ (note that $x$ is already negative, and by simplifying the left inequality, we see that such an $n$ and $m$ can always be chosen), we guarantee that such an $(x,y)$ produces the desired values of $\g$ and $\gst$.

Finally, consider the three-genus.  Suppose K as calculated above is a knot (in particular $x$ and $y$ are integers).  It may be that $g_3(K) \neq l$.  If necessary, we may lower the three-genus to $k$, by definition of $\g(K)$ (without changing $\g$ or $\gst$).  Let $K' \sim K$ be such that $g_3(K') = \g(K) = k$.  Let $J$ be a slice knot with $g_3(J) = 1$ (for instance $6_1$).  The three-genus is additive, so $g_3(xJ) = x$.  Then if $K'' = K' \# (l-k) J$, we have $g_3(K'') = g_3(K') + g_3((l-k)J) = k+ (l-k) = l$.  Since $J$ was slice, we still have $\g(K'') = k$ and $\gst(K'') = j$.  
\end{proof}

\begin{note}
In the proof above, we have not required (and in fact usually may not assume) that $x$ and $y$ be integer values.  Hence, we have only completed the proof in $\Cq$ and only claim it to be true for the stable invariants, not the concordance genus and four genus (since these are not defined for $\Cq$).  A more detailed examination of the stable four genus may yield different examples in which we can demand integer values of our coefficients.  In this case, we can refine the result to give an actual knot.
\end{note}

\begin{conj}
For any $j$, $k$, and $l \in \mathbb{Z}$, for which $1 \leq j \leq k \leq l$, there is some knot K for which $g_4(K) = j$, $g_c(K) = k$, and $g_3(K) = l$. 
\end{conj}

It has been previously observed in work by Casson and also in work of Nakanishi \cite{N:Unknotting} that the gap between $g_4(K)$ and $g_c(K)$ can be made arbitrarily large.  A proof of this conjecture would confirm that we can additionally construct $K$ to have a given value for $g_c(K)$.  

There are many other open questions raised by this invariant.  We'll conclude with listing several of them.
\begin{itemize}
\item In the examples in the previous section 4 we observed that $\gst(K) = g_4(K)$ and $\g(K) = g_c(K)$.  We saw in section 3 that this is not always the case.  It is unknown whether this gap can be made arbitrarily large.
\item Livingston gives an example in \cite{L:Stable4Genus} of a knot with rational (non-integer) stable four genus.  On the other hand, there are no known knots with rational (non-integer) stable concordance genus.
\item In all of the examples calculated in this paper, if $\g(K) = k$, then for some integer multiple of $K$, $\frac{g_c(nK)}{n} = k$.  Of course, it is not necessarily true that a limit of a sequence must appear in that sequence.  It is an open question whether there is a knot $K$ for which $\g(K) = k$, but there is no multiple $n$ of $K$ such that $\frac{g_c(nK)}{n} = k$.
\item A special case of the previous question: Does there exist a knot $K$ which is not finite order in the concordance group but $\g(K) = 0$?  It is clearly true that if $K$ is torsion in the concordance group, then $\g(K) = 0$.  If the converse is true, it could prove to be a very useful tool to identify torsion in the concordance group.  It is known that there is two-torsion in the concordance group, but it is still unknown whether there is any other torsion in the concordance group.
\item We observed that if $g_4(K) = 0$ then $g_c(K) = 0$.  Does the same hold for $\gst$ and $\g$?  This is in fact related to the previous question.  If it is true that $\g(K) = 0$ only when $K$ is torsion in the concordance group, and similarly that $\gst(K) = 0$ only when $K$ is torsion in the concordance group, then it must also be true that whenever $\gst(K) = 0$ then $\g(K) = 0$ as well.  Otherwise, there may be a distinction between the stable invariants which cannot arise for the classical invariants.
\end{itemize}

\bibliographystyle{hacm}
\bibliography{MKKrefs}

\begin{thebibliography}{1}

\bibitem{knotinfo}
{\sc Cha, J.~C., and Livingston, C.}
\newblock {Knot{I}nfo: Table of Knot Invariants}.
\newblock \url{http://www.indiana.edu/~knotinfo}, October 2013.

\bibitem{FM:CobordismOfKnots}
{\sc Fox, R.~H., and Milnor, J.~W.}
\newblock Singularities of {$2$}-spheres in {$4$}-space and cobordism of knots.
\newblock {\em Osaka J. Math. 3\/} (1966), 257--267.

\bibitem{G:JonesSignature}
{\sc {Garoufalidis}, S.}
\newblock {Does the Jones polynomial determine the signature of a knot?}
\newblock {\em ArXiv Mathematics e-prints\/} (Oct. 2003), arXiv:math/0310203.

\bibitem{KK:11Crossing}
{\sc {Kearney}, M.~K.}
\newblock {The Concordance Genus of 11--Crossing Knots}.
\newblock {\em J. Knot Theory Ramifications\/} (Aug. 2012),
  arXiv:math/1208.5059.
\newblock (to appear).

\bibitem{L:Stable4Genus}
{\sc Livingston, C.}
\newblock The stable 4-genus of knots.
\newblock {\em Algebr. Geom. Topol. 10}, 4 (2010), 2191--2202.

\bibitem{M:SignatureInvariants}
{\sc Matumoto, T.}
\newblock On the signature invariants of a non-singular complex sesquilinear
  form.
\newblock {\em J. Math. Soc. Japan 29}, 1 (1977), 67--71.

\bibitem{M:InfiniteCyclicCoverings}
{\sc Milnor, J.~W.}
\newblock Infinite cyclic coverings.
\newblock In {\em Conference on the {T}opology of {M}anifolds ({M}ichigan
  {S}tate {U}niv., {E}. {L}ansing, {M}ich., 1967)}. Prindle, Weber \& Schmidt,
  Boston, Mass., 1968, pp.~115--133.

\bibitem{N:Unknotting}
{\sc Nakanishi, Y.}
\newblock A note on unknotting number.
\newblock {\em Math. Sem. Notes Kobe Univ. 9}, 1 (1981), 99--108.

\bibitem{Rolfsen}
{\sc Rolfsen, D.}
\newblock {\em Knots and Links}, vol.~7 of {\em Mathematics Lecture Series}.
\newblock Publish or Perish Inc., Houston, TX, 1990.
\newblock Corrected reprint of the 1976 original.

\end{thebibliography}

\end{document}